\documentclass[twoside,leqno,twocolumn]{article}
\pdfoutput=1
\usepackage{ltexpprt-new}
\usepackage{mathptmx}
\usepackage{microtype}
\usepackage{url}
\usepackage{amssymb,amsmath}
\usepackage{amsmath,amsthm,amsfonts,amssymb,amscd}

\usepackage{fancyhdr}
\renewcommand{\footskip}{50pt}
\newtheoremstyle{sctheorem} % name
    {\topsep}                    % Space above
    {\topsep}                    % Space below
    {\itshape}                   % Body font
    {}                           % Indent amount
    {\scshape}                   % Theorem head font
    {.}                          % Punctuation after theorem head
    {.5em}                       % Space after theorem head
    {}  % Theorem head spec (can be left empty, meaning 'normal')

\newtheoremstyle{scdefinition} % name
    {\topsep}                    % Space above
    {\topsep}                    % Space below
    {}                           % Body font
    {}                           % Indent amount
    {\scshape}                   % Theorem head font
    {.}                          % Punctuation after theorem head
    {.5em}                       % Space after theorem head
    {}  % Theorem head spec (can be left empty, meaning 'normal')

\theoremstyle{sctheorem}
\newtheorem{lemma}{Lemma}
\newtheorem{theorem}[lemma]{Theorem}
\newtheorem{corollary}[lemma]{Corollary}
\newtheorem{proposition}[lemma]{Proposition}
\newtheorem{conjecture}[lemma]{Conjecture}
\theoremstyle{scdefinition}
\newtheorem{definition}[lemma]{Definition}
\newtheorem{example}[lemma]{Example}
\numberwithin{lemma}{section}
\numberwithin{equation}{section}

\newcommand\C{{\mathbb{C}}}
\newcommand\R{{\mathbb{R}}}

\newcommand\Z{{\mathbb{Z}}}

\newcommand\supp{\mathop{\textup{supp}}}
\newcommand\Sym{\mathop{\textup{Sym}}}

\newcommand{\vr}{\hat{r}}
\newcommand{\vs}{\hat{s}}
\newcommand{\vt}{\hat{t}}
\newcommand{\vu}{\hat{u}}
\newcommand{\vv}{\hat{v}}
\newcommand{\vw}{\hat{w}}
\newcommand{\vx}{\hat{x}}
\newcommand{\vy}{\hat{y}}
\newcommand{\vz}{\hat{z}}

\newcommand\ind[1]{_{#1}}

\begin{document}

\title{\Large Fast matrix multiplication using coherent
configurations}
\author{Henry Cohn\thanks{Microsoft Research New England, One Memorial Drive,
Cambridge, MA 02142,
\texttt{cohn@microsoft.com}} \\
\and Christopher Umans\thanks{Computing + Mathematical Sciences,
California Institute of Technology, Pasadena, CA 91125,
\texttt{umans@cms.caltech.edu}; research supported by NSF grants
CCF-0846991 and CCF-1116111 and BSF grant 2010120.}}
\date{}

\maketitle

\pagenumbering{arabic}
\setcounter{page}{1074}

\begin{abstract}
%\small\baselineskip=9pt
We introduce a relaxation of the notion of tensor rank, called
\emph{$s$-rank}, and show that upper bounds on the $s$-rank of the
matrix multiplication tensor imply upper bounds on the ordinary rank.
In particular, if the ``$s$-rank exponent of matrix multiplication''
equals 2, then $\omega = 2$.  This connection between the $s$-rank
exponent and the ordinary exponent enables us to significantly
generalize the group-theoretic approach of Cohn and Umans, from group
algebras to general algebras. Embedding matrix multiplication into
general algebra multiplication yields bounds on $s$-rank (not ordinary
rank) and, prior to this paper, that had been a barrier to working with
general algebras.

We identify \emph{adjacency algebras} of \emph{coherent configurations}
as a promising family of algebras in the generalized framework.
Coherent configurations are combinatorial objects that generalize
groups and group actions; adjacency algebras are the analogue of group
algebras and retain many of their important features. As with groups,
coherent configurations support matrix multiplication when a natural
combinatorial condition is satisfied, involving triangles of points in
their underlying geometry.

Finally, we prove a closure property involving symmetric powers of
adjacency algebras, which enables us to prove nontrivial bounds on
$\omega$ using \emph{commutative} coherent configurations and suggests
that commutative coherent configurations may be sufficient to prove
$\omega = 2$. Altogether, our results show that bounds on $\omega$ can
be established by embedding large matrix multiplication instances into
small commutative coherent configurations.
\end{abstract}

\section{Introduction}

Determining the exponent of matrix multiplication is one of the most
fundamental unsolved problems in algebraic complexity. This quantity is
the smallest number $\omega$ such that $n \times n$ matrix
multiplication can be carried out using $n^{\omega+o(1)}$ arithmetic
operations as $n$ tends to infinity.  Clearly $\omega \ge 2$, and it is
widely believed that $\omega=2$, but the best upper bound known is
$\omega \le 2.3727$ (due to Vassilevska Williams \cite{W}). The
importance of $\omega$ is by no means limited to matrix multiplication,
as $\omega$ also describes the asymptotic complexity of many other
problems in linear algebra and graph theory (see Chapter~16 of
\cite{BCS}).

In the 43 years since Strassen's original paper \cite{S} gave the first
improvement on the obvious exponent bound $\omega \le 3$, there have
been several major conceptual advances in the effort to obtain upper
bounds on $\omega$, each of which can informally be understood as
relaxing the ``rules of the game.'' For example, Bini \cite{B} showed
that an upper bound on the \emph{border rank} of a tensor implies an
upper bound on its asymptotic rank. Indeed, there are useful examples
of tensors with border rank strictly smaller than their rank, which led
to improvements over Strassen's original algorithm. Sch\"onhage
\cite{Sch} showed how to convert upper bounds on the rank of the direct
sum of several matrix multiplication tensors into an upper bound on
$\omega$, and his \emph{asymptotic sum inequality} has played a crucial
role in nearly all further advances. Strassen's \emph{laser method}
\cite{S87} gave a way to convert non-matrix multiplication tensors
(whose coarse structure contains a large diagonal, and whose components
are all isomorphic to matrix multiplication tensors) into upper bounds
on $\omega$, and this method was used by Coppersmith and Winograd
\cite{CW} as well as in the recent improvements of Davie and Stothers
\cite{St,DS} and Vassilevska Williams \cite{W}.

Here we introduce a further relaxation of the rules of the game, by
studying a weighted version of matrix multiplication.  Instead of
computing the product $AB$ of two matrices via
\[
(AB)_{i,k} = \sum_j A_{i,j} B_{j,k},
\]
we use
\[
\sum_j \lambda_{i,j,k} A_{i,j} B_{j,k},
\]
where the coefficients $\lambda_{i,j,k}$ are nonzero complex numbers.
Of course, in certain cases weighted matrix multiplication is trivially
equivalent to ordinary matrix multiplication.  For example, if
$\lambda_{i,j,k}$ can be written as $\alpha_{i,j} \beta_{j,k}
\gamma_{k,i}$, then weighted matrix multiplication amounts to ordinary
multiplication of matrices whose entries have been rescaled.  However,
rescaling does not yield an efficient equivalence for arbitrary
weights.

We capture the complexity of weighted matrix multiplication via a new
exponent $\omega_s$, satisfying $2 \le \omega_s \le \omega$.  It is the
smallest real number for which there exist weights (depending on the
dimensions of the matrices) such that weighted $n \times n$ matrix
multiplication can be carried out in $n^{\omega_s+o(1)}$ arithmetic
operations.  The ``s'' stands for ``support,'' because we are dealing
with tensors that have the same support as the matrix multiplication
tensors.

In \cite{CU}, we showed how to embed matrix multiplication into group
algebra multiplication, and this methodology was used in \cite{CKSU} to
prove strong bounds on $\omega$.  Replacing group algebras with more
general algebras has always been an appealing generalization, and
indeed the same approach works, except that it yields an embedding of
\emph{weighted} matrix multiplication.  Thus, it gives an upper bound
on $\omega_s$, rather than $\omega$. Prior to this paper, an upper
bound on $\omega_s$ was of interest only by analogy with $\omega$, and
it was not known to imply anything about $\omega$ itself. Here, we
overcome this obstacle by bounding $\omega$ in terms of $\omega_s$, and
we develop this embedding approach for a promising class of algebras.

Our main results are:

\begin{enumerate}
\item[(1)] We prove that $\omega \le (3\omega_s-2)/2$.  In
    particular, if $\omega_s \le 2+\varepsilon$, then
    $\omega \le 2+ (3/2)\varepsilon$, so bounds for
    $\omega_s$ can be translated into bounds for $\omega$
    with just a 50\% penalty.  Of course, that penalty is
    significant when $\varepsilon$ is large, but our bound
    makes weighted matrix multiplication a viable approach
    for proving that $\omega=2$ (as then $\varepsilon=0$).
    This inequality between $\omega$ and $\omega_s$ can be
    proved using the laser method, but it does not seem to
    have been observed previously. We give a direct and
    self-contained proof in Section~\ref{section:bounds} as
    well as an explanation via the laser method. We also
    show that Boolean matrix multiplication has a
    randomized algebraic algorithm with running time
    $n^{\omega_s + o(1)}$, which avoids the 50\% penalty.

\item[(2)] We identify adjacency algebras of coherent configurations as
    a promising family of algebras.  Coherent configurations are
    combinatorial objects that generalize groups and group actions;
    adjacency algebras are the analogue of group algebras and
    retain many of their important features. In particular, each
    adjacency algebra possesses a basis corresponding to an
    underlying geometry, and weighted matrix multiplication can be
    embedded when the coherent configuration satisfies a
    combinatorial condition involving triangles of points.

\item[(3)] We prove a fundamental closure property of this class of
    algebras: any bound on $\omega_s$ obtained by applying the asymptotic sum
    inequality to independent embeddings of several weighted matrix
    multiplications can also be proved using a single embedding
    into a \emph{symmetric power} of the algebra.  Symmetric powers of
    adjacency algebras are themselves adjacency algebras, and this
    operation also preserves commutativity. Our results open the
    possibility of achieving $\omega=2$ using commutative adjacency
    algebras, and we conjecture that commutative adjacency algebras
    suffice.  In fact, that would follow from either of the two
    conjectures in \cite{CKSU}.
\end{enumerate}

A simple pigeonhole principle argument (Lemma~3.1 in \cite{CU}) shows
that one cannot nontrivially embed a single matrix multiplication problem into a
commutative group algebra.  One might expect a similar barrier for
commutative adjacency algebras, but the pigeonhole argument breaks down
in this setting.  Indeed, in this paper we prove nontrivial bounds on
$\omega$ using commutative adjacency algebras in
Theorem~\ref{thm:cksu-bounds-on-omega}, by applying our machinery to the
constructions from \cite{CKSU} (although we do not improve on the best
known bounds). We should note that the simultaneous triple product
property from \cite{CKSU} previously showed that one could avoid
noncommutativity at the cost of having to deal with several independent
embeddings.  One could return to the setting of a single embedding
using the wreath product construction (Theorem~7.1 in \cite{CKSU}), but
this reintroduced noncommutativity, whereas working with coherent
configurations rather than groups, as we do in this paper, avoids it completely.

The advantage of commutativity is that obtaining exponent bounds then
amounts to a familiar type of task: embed as large an object as possible
(here, a matrix multiplication instance) into as small an object as possible
(here, a coherent
configuration, with ``size'' measured by rank).  By contrast, the
noncommutative case involves a third quantity, namely the dimensions of
the irreducible representations of the algebra.

\section{Preliminaries and background}

We define $[n] = \{1,2,\dots,n\}$.

\subsection{Tensors}

Our results will all be stated in terms of tensors.  Recall that
tensors are a generalization of vectors and matrices to higher orders.
Tensor products of vector spaces form an elegant algebraic setting for
the theory of tensors, but we will adopt the more concrete approach of
representing tensors as multilinear forms. For example, the matrix with
entries $A_{i,j}$ corresponds to the bilinear form $\sum_{i,j} A_{i,j}
\vx_i \vy_j$, where $\vx_i$ and $\vy_j$ are formal variables, and we
can represent a third-order tensor as $\sum_{i,j,k} A_{i,j,k} \vx_i
\vy_j \vz_k$. We will use hats to make it clear which symbols denote
formal variables.  Applying invertible linear transformations to the
sets of variables (here, $\{\vx_i\}$, $\{\vy_j\}$, and $\{\vz_k\}$)
yields an isomorphic tensor, but we cannot mix variables from different
sets.

The \emph{direct sum} $T \oplus T'$ of two tensors is simply their sum,
if they have no variables in common (otherwise, first change variables
to remove any overlap). For the \emph{tensor product} $T \otimes T'$,
if $T = \sum_{i,j,k} T_{i,j,k} \vx_i \vy_j \vz_k$ and $T' =
\sum_{\ell,m,n} T'_{\ell,m,n} \vu_\ell \vv_m \vw_n$, then
\[
T \otimes T' = \sum_{i,j,k,\ell,m,n} T_{i,j,k} T'_{\ell,m,n} \vr_{i,\ell} \vs_{j,m} \vt_{k,n},
\]
with new variables $\vr_{i,\ell}$, $\vs_{j,m}$, and $\vt_{k,n}$.  In
other words, we simply take the product of $T$ and $T'$ but combine the
variables as illustrated above (e.g., $\vx_i\vu_\ell$ becomes
$\vr_{i,\ell}$). The direct sum and tensor product are defined only for
tensors of the same order, and they preserve that order.

The \emph{rank} $R(T)$ of a tensor $T$ is one of its most important
invariants. A nonzero tensor has rank $1$ if it is the product of
linear forms, and rank $r$ if it is the sum of $r$ rank $1$ tensors but
no fewer. In other words, $\sum_{i,j,k} T_{i,j,k} \vx_i \vy_j \vz_k$
has rank at most $r$ if there are linear forms $\alpha_\ell(\vx)$,
$\beta_\ell(\vy)$, and $\gamma_\ell(\vz)$ such that
\[
\sum_{i,j,k} T_{i,j,k} \vx_i \vy_j \vz_k = \sum_{\ell=1}^r \alpha_\ell(\vx) \beta_\ell(\vy) \gamma_\ell(\vz).
\]
Tensor rank generalizes the concept of matrix rank, but it is more
subtle. While matrices can be brought into a simple canonical form (row
echelon form) in which their rank is visible, tensors cannot, because
the symmetry group acting on them has far too low a dimension compared
with the dimension of the space of tensors itself. Indeed, computing
tensor rank is NP-hard \cite{H}.

\subsection{Matrix multiplication in terms of tensors}

The \emph{matrix multiplication tensor} $\langle \ell,m,n \rangle$ is
the tensor
\[
\sum_{i=1}^{\ell} \sum_{j=1}^m \sum_{k=1}^n \vx_{i,j} \vy_{j,k} \vz_{k,i}.
\]
Note that the coefficient of $\vz_{k,i}$ singles out the
$\vx_{i,j}\vy_{j,k}$ terms that occur in the $(i,k)$ entry of the
matrix product.  It is easy to check that $\langle \ell,m,n \rangle
\otimes \langle \ell',m',n' \rangle \cong \langle \ell\ell',mm',nn'
\rangle$, which amounts to the assertion that block matrix multiplication computes the matrix product.

A low-rank expression for $\langle \ell,m,n \rangle$ specifies an
efficient \emph{bilinear algorithm} for computing the product of $\ell
\times m$ and $m \times n$ matrices.  In particular, it follows that
$(\ell m n)^{\omega/3} \le R(\langle \ell,m,n \rangle)$
(Proposition~15.5 in \cite{BCS}).

In fact, although we have defined $\omega$ in terms of arbitrary
algebraic algorithms, it is completely characterized by rank via
\[
\omega = \inf\{\tau \in \R : R(\langle n,n,n \rangle) =
O(n^\tau)\}
\]
(Proposition~15.1 in \cite{BCS}). In other words, bilinear algorithms
have the same exponent as arbitrary algebraic algorithms.  Thus, the
entire subject of fast matrix multiplication can be reduced to bounding
the rank of matrix multiplication tensors.

Sch\"onhage's \emph{asymptotic sum inequality} \cite{Sch} states that
\begin{multline} \label{eqn:asi}
(\ell_1 m_1 n_1)^{\omega/3} + \dots + (\ell_k m_k n_k)^{\omega/3} \\
\le R\big(\langle \ell_1,m_1,n_1 \rangle \oplus \dots \oplus \langle
\ell_k,m_k,n_k\rangle\big),
\end{multline}
and furthermore that the same holds for border rank (a relaxation of
the notion of rank which will not play an important role in this
paper). Thus, an unexpectedly efficient method for carrying out several
independent matrix multiplications yields a bound on $\omega$.

See \cite{BCS} for further background on tensors, matrix
multiplication, and algebraic complexity in general.  It is important
to keep in mind that the tensor manipulations all have implicit
algorithms behind them.  In principle one could dispense with the
tensor formalism completely, but it plays a valuable role in focusing
attention on the central issues.

\section{Matrix multiplication exponent bounds via s-rank}
\label{section:bounds}

In this section, we show that an upper bound on what we call the
``support rank''---or \emph{s-rank}---of a matrix multiplication tensor
implies an upper bound on $\omega$.  The \emph{support} $\supp(T)$ of a
tensor $T$ is the set of monomials that have nonzero coefficients.  Of
course this depends on the choice of basis and is therefore not an
isomorphism invariant, and the same is true for concepts like s-rank
that are defined in terms of it. However, basis dependence is not a
difficulty in algebraic complexity. After all, any computational
problem must specify a choice of basis for use in input and output, and
writing a tensor as a multilinear form already involves an implicit
choice of basis (the choice of variables).

\begin{definition}
The \emph{s-rank} $R_s(T)$ of a tensor $T$ is the minimum rank of
a tensor $T'$ for which $\supp(T) = \supp(T')$.
\end{definition}

Clearly s-rank can be no larger than rank. Here is a simple example
that shows that s-rank can be dramatically smaller than both rank
and border rank:

\begin{proposition}
The $n \times n$ matrix $J - I$, where $J$ is the all ones matrix and
$I$ is the identity matrix, has rank $n$ and border rank $n$, but
s-rank equal to $2$.
\end{proposition}

\begin{proof}
Rank and border rank coincide for matrices (the matrices of rank at
most $r$ are characterized by determinantal conditions and thus form a
closed set), and $J-I$ has rank $n$. However, consider the rank $1$
matrix $M$ defined by $M_{i,j} = \zeta^{i-j}$, where $\zeta$ is a
primitive $n$-th root of unity. Then $M - J$ has the same support as
$J-I$. Because $M$ and $J$ are both rank 1 matrices, the s-rank of
$J-I$ is at most $2$. It is also easy to see that no matrix with the
same support as $J-I$ has rank 1, so the s-rank is exactly 2.
\end{proof}

On the other hand, border rank can be smaller than s-rank, so
these two relaxations of rank are incomparable:
\begin{proposition}
\label{prop:border-rank-smaller-than-s-rank}
The tensor $T = \vx_0\vy_0\vz_0 + \vx_0\vy_1\vz_1 + \vx_1\vy_0\vz_1$ has border rank 2 and s-rank 3.
\end{proposition}
\begin{proof}
We refer to Bl\"{a}ser's notes \cite[p. 31]{Blaser} for the
simple proof that the border rank is 2. To show that the s-rank
is at least 3, we mimic his proof via the substitution method
that the (ordinary) rank is 3. In a decomposition of any tensor
$T'$ with the same support as $T$ into the sum of rank one
tensors, one of the rank one tensors must depend on $\vx_1$. We
can make this tensor zero by substituting a scalar multiple of
$\vx_0$ for $\vx_1$. After this substitution, $T'$ still
depends on $\vy_1$, so there is another rank one tensor in the
decomposition that depends on $\vy_1$. We can make this tensor
zero by substituting a scalar multiple of $\vy_0$ for $\vy_1$.
After both substitutions, $T'$ still depends on $\vz_0$,
so there must be at least one more rank one tensor in
the decomposition. The corresponding s-rank upper bound of 3 is
trivial.
\end{proof}

Like ordinary rank, s-rank is subadditive and submultiplicative in the
sense that for tensors $T$ and $T'$, we have $R_s(T \oplus T') \le
R_s(T) + R_s(T')$ and $R_s(T \otimes T') \le R_s(T)R_s(T')$. For matrix
multiplication tensors, we have $R_s(\langle \ell,m,n \rangle) =
R_s(\langle \ell',m',n' \rangle)$ for every permutation $(\ell',m',n')$
of $(\ell,m,n)$. In analogy to the exponent of matrix multiplication
$\omega$, we define $\omega_s$, the s-rank exponent of matrix
multiplication, as follows:

\begin{definition}
The \emph{s-rank exponent of matrix multiplication}, denoted
$\omega_s$, is defined by
\[\omega_s = \inf\{\tau \in \R : R_s(\langle n,n,n \rangle) =
O(n^\tau)\}.\]
\end{definition}

By comparison, the exponent $\omega$ can be defined in the same way,
with ordinary rank replacing s-rank in the above expression. Since
every tensor having the same support as $\langle n,n,n \rangle$ has
$n^2$ linearly independent slices, $R_s(\langle n,n,n \rangle) \ge
n^2$, and thus $2 \le \omega_s \le \omega$.

As one would expect, an $s$-rank upper bound implies an upper bound on
$\omega_s$.  Here is the $s$-rank version of the standard proof:

\begin{proposition}
\label{prop:s-rank-omega-bound} For all $\ell,m,n$, we have
\[
(\ell m
n)^{\omega_s/3} \le R_s(\langle \ell,m,n \rangle).
\]
\end{proposition}

\begin{proof}
Let $r = R_s(\langle \ell,m,n \rangle)$ and $M = \ell m n$. By
symmetrizing, we have $R_s(\langle M,M,M \rangle)\le r^3$, and then for
all $N \ge 1$, by padding to the next largest power $M^i$ of $M$,
\begin{align*}
R_s(\langle N,N,N \rangle) &\le R_s(\langle M^i,M^i,M^i \rangle)\\
&\le r^{3i}\\
&= \big(M^i\big)^{3 \log_M r}\\
&= O\big(N^{3\log_M r}\big).
\end{align*}
Thus, $\omega_s \le 3\log_M r$, from which the theorem follows.
\end{proof}

We note that one can define the border s-rank of $T$ to be the minimum
border rank of tensors with the same support as $T$, and then by using
Bini's argument \cite{B}, the above proposition holds with border
s-rank in place of s-rank.

Whereas an upper bound on the rank of a matrix multiplication tensor
implies a bilinear algorithm for matrix multiplication, an upper bound
on the s-rank implies a bilinear algorithm for a \emph{weighted}
version of matrix multiplication: given matrices $A$ and $B$, the
algorithm computes values
\[
C_{i,k} = \sum_j \lambda_{i,j,k}A_{i,j}B_{j,k},
\]
where the weights $\lambda_{i,j,k}$ are certain nonzero scalars
(depending on the construction used to attain a low s-rank). In other
words, each entry of the result matrix $C$ is a weighted inner product,
with different weightings for the different inner products.  There
seems to be no obvious transformation to remove these weights.

As noted above, $2 \le \omega_s \le \omega$, so upper bounds on
$\omega$ imply upper bounds on $\omega_s$ (and $\omega = 2$ implies
$\omega_s = 2$). Theorem~\ref{thm:remove-weights} below shows that upper
bounds on $\omega_s$ imply upper bounds on $\omega$, and indeed
$\omega_s = 2$ implies $\omega = 2$. Thus s-rank is a useful relaxation
of rank, when trying to bound the exponent of matrix multiplication.

\begin{theorem}
\label{thm:remove-weights} The exponents $\omega$ and
$\omega_s$ satisfy
\[
\omega \le (3\omega_s - 2)/2.
\]
\end{theorem}

In other words, $\omega_s \le 2 + \varepsilon$ implies $\omega \le 2 +
(3/2)\varepsilon$.

\begin{proof}
By the definition of $\omega_s$, we have $R_s(\langle n, n, n \rangle)
= n^{\omega_s+o(1)}$.  For a given value of $n$, let $T$ be the
trilinear form corresponding to this (weighted) $n \times n$ matrix
multiplication:
\[
T = \sum_{a,b,c \in [n]} \lambda_{a,b,c}\vx\ind{a,b}\vy\ind{b,c}\vz\ind{c,a}
\]
with $0 \ne \lambda_{a,b,c}\in \C$ for all $a,b,c$. Let
\[S
\subseteq
\Delta_n = \{(s_1,s_2,s_3) : s_1,s_2,s_3 \in [n] \mbox{ and } s_1+ s_2
+ s_3 = n+2\}
\]
be a \emph{triangle-free set}, as defined in Section~6.2 of
\cite{CKSU}. Such a set has the property that if $s,t,u \in S$ satisfy
$s_1 = t_1$, $t_2 = u_2$, and $u_3 = s_3$ then $s = t = u$. In
\cite{CKSU} we gave a simple construction of triangle-free sets $S$
with $|S| = n^{2 - o(1)}$. Let $T'$ be the trilinear form corresponding
to $|S|$ independent $n^2 \times n^2$ matrix multiplications; i.e.,
\[
T' = \sum_{s \in S, \ i,j,k \in [n]^2} \vu\ind{s, i,j}\vv\ind{s, j,k}\vw\ind{s, k,i}.
\]

We will show that $T'$ is a restriction of the tensor power $T^{\otimes
3}$, which is given by
\[
T^{\otimes 3} = \sum_{a,b,c \in [n]^3} \lambda_{a_1,b_1,c_1}
\lambda_{a_2,b_2,c_2} \lambda_{a_3,b_3,c_3} \vx_{a,b} \vy_{b,c} \vz_{c,a}.
\]
In other words, we will show that $T'$ can be obtained by substituting variables
in $T^{\otimes 3}$, which implies that $R(T') \le R\big(T^{\otimes 3}\big)$.
To do so, define (for $s,t,u \in S$ and $i, i', j, j', k, k' \in [n]^2$)
\begin{align*}
\vu\ind{s,i,j'} &=  \lambda_{i_2, j_1', s_2} \vx\ind{(i_1, i_2, s_3), (s_1, j_1', j_2')}  \\
\vv\ind{t,j,k'} &=  \lambda_{t_3, j_2, k_2'} \vy\ind{(t_1, j_1, j_2), (k_1', t_2, k_2')}  \\
\vw\ind{u,k,i'} &=  \lambda_{i_1', u_1, k_1} \vz\ind{(k_1, u_2, k_2), (i_i', i_2', u_3)}
\end{align*}
and set the $\vx,\vy,\vz$ variables not mentioned in these equations
equal to zero. Under this change of variables, we will see that
$T^{\otimes 3}$ becomes exactly the tensor $T'$.  To check this, we
must verify that upon substituting the $\vu, \vv, \vw$ variables for
the $\vx, \vy, \vz$ variables in $T^{\otimes 3}$ according to the above
formulas, the coefficient of
$\vu\ind{s,i,j'}\vv\ind{t,j,k'}\vw\ind{u,k,i'}$ is $1$ if $s=t=u$,
$i=i'$, $j=j'$, and $k=k'$, and it is $0$ otherwise. Since the support
of $T^{\otimes 3}$ is the same as the support of $\langle n^3, n^3, n^3
\rangle$, the monomial
\[
\vx\ind{(i_1, i_2, s_3), (s_1, j_1', j_2')}\vy\ind{(t_1, j_1, j_2),
(k_1', t_2, k_2')}\vz\ind{(k_1, u_2, k_2), (i_i', i_2', u_3)}
\] in $T^{\otimes 3}$
has a nonzero coefficient if and only if
\begin{align*}
(s_1, j_1', j_2') &=  (t_1, j_1, j_2)\\
(k_1', t_2, k_2') &=  (k_1, u_2, k_2)\\
(i_i', i_2', u_3) &=  (i_1, i_2, s_3).
\end{align*}
This happens if and only if $i=i'$, $j=j'$, $k=k'$, $s_1 = t_1$, $t_2 =
u_2$, and $u_3 = s_3$, and by the definition of a triangle-free set the
last three conditions imply $s=t=u$.  The coefficient in
$T^{\otimes 3}$ of
\[
\vx\ind{(i_1,i_2,s_3),(s_1,j_1,j_2)}\vy\ind{(s_1,j_1,j_2),(k_1,s_2,k_2)}\vz\ind{(k_1,s_2,k_2),(i_1,i_2,s_3)}
\]
is $\lambda_{i_1,s_1,k_1} \lambda_{i_2,j_1,s_2} \lambda_{s_3,j_2,k_2}$,
which exactly corresponds to the $\lambda_{i_2, j_1', s_2}\lambda_{t_3,
j_2, k_2'}\lambda_{i_1', u_1, k_1}$ factor from the definition of
$\vu,\vv,\vw$, so the coefficient of
$\vu\ind{s,i,j}\vv\ind{s,j,k}\vw\ind{s,k,i}$ after the substitution is $1$.

Thus the ordinary rank of the direct sum of $|S| = n^{2 - o(1)}$
independent $n^2 \times n^2$ matrix multiplications is at most
$\big(n^{\omega_s + o(1)}\big)^3$, and applying the asymptotic sum
inequality \eqref{eqn:asi}, we get
\[n^{2 - o(1)}n^{2 \omega} \le (n^{\omega_s + o(1)})^3.\]
Taking logarithms and letting $n$ go to infinity, we get $2 + 2 \omega
\le 3\omega_s$, as desired.
\end{proof}

Theorem~\ref{thm:remove-weights} can also be proved using the laser
method, in particular using Proposition~7.3 from \cite{S87}
(Proposition 15.32 in \cite{BCS}), although this consequences does not
seem to have been observed in the literature. Here is a sketch of the
proof. By definition, there exist tensors having the same support
as $\langle n,n,n \rangle$, with rank $n^{\omega_s + o(1)}$.
Adopting the language of \cite{BCS}, such a tensor has a direct sum
decomposition $D$ whose $D$-support is isomorphic to $\langle
1,n,1\rangle$ and whose $D$-components are each \emph{isomorphic} to
$\langle n,1,n \rangle$ (it is a special feature of an outer-product
tensor, which has only a single 1 in each slice, that all tensors with
the same support are isomorphic). Proposition~15.32 in \cite{BCS} then
implies that $n^2n^{2\omega} \le \big(n^{\omega_s + o(1)}\big)^3$,
which yields the same bound as Theorem~\ref{thm:remove-weights}.

It is an interesting open problem to improve the conclusion of Theorem
\ref{thm:remove-weights} to $\omega \le \omega_s$. In the preceding
paragraph, the $D$-component isomorphism being used is very special (it
corresponds to a \emph{diagonal} change of basis), so one might hope to
avoid the loss coming from machinery that handles arbitrary
isomorphisms.

In fact, this loss \emph{can} be avoided entirely when one is
interested in the simpler problem of Boolean matrix
multiplication, i.e., matrix multiplication over the Boolean
semiring with ``and'' as multiplication and ``or'' as addition.
The next theorem describes how to use weighted multiplication
directly to obtain an algebraic algorithm for Boolean matrix
multiplication.

Given $n \times n$ Boolean matrices $A$ and $B$, let $A'$ be the obvious
lift of $A$ to a $0/1$ complex matrix, and define $B'$ the same way but
with a random choice of $1$ or $2$ for each nonzero entry.

\begin{theorem}
There is an algebraic
algorithm running in $n^{\omega_s + o(1)}$ operations that
computes from $A', B'$ an $n \times n$ matrix $C$ with the following
property: the $(i,j)$ entry of $C$ is $0$ if the $(i,j)$ entry of
the Boolean matrix product of $A$ and $B$ is $0$; otherwise it is
nonzero with probability at least $1/2$.
\end{theorem}

The procedure may be repeated $O(\log n)$ times to obtain all
entries of the Boolean product of $A$ and $B$ with high
probability.

\begin{proof}
Using a bilinear algorithm, one can compute a suitably weighted
product of $A'$ and $B'$ in $n^{\omega_s + o(1)}$ operations.
We get a result matrix whose $(i,j)$
entry is
\[\sum_{\ell} \lambda_{i,\ell,j} A_{i,\ell}'B_{\ell, j}',\]
where $\lambda_{i,\ell,j} \ne 0$.
When the $(i,j)$ entry of the Boolean matrix product of $A$ and
$B$ is zero, this value is clearly also zero; otherwise, it equals
$\sum_{\ell \in L}\lambda_{i,\ell, j}r_\ell$ for a nonempty set
$L$ and each $r_\ell$ chosen randomly from $\{1,2\}$. For a
given $\ell \in L$, there is a unique value of $r_\ell$ making
this sum zero, so the probability that it vanishes is at most
$1/2$.
\end{proof}

If the weights arising in the above proof are all
positive (as they are for all the $s$-rank bounds we derive in
this paper), then no randomness is needed, as $A'$ and $B'$ can
both be taken to be the obvious lifts of $A$ and $B$ to $0/1$
matrices.

Finally, we note that all of the manipulations used in the matrix
multiplication literature for converting bounds on the rank of certain
``basic'' tensors into bounds on the rank of the matrix multiplication
tensor also work with s-rank in place of rank. So, for example, s-rank
bounds on the partial matrix multiplication tensor of Bini, Capovani,
Lotti, and Romani \cite{BCLR}, the basic tensor used by Sch\"onhage
\cite{Sch}, the basic tensor in Strassen's laser method paper
\cite{S87}, or any of the basic tensors introduced by Coppersmith and
Winograd \cite{CW} eventually yield an s-rank bound on a matrix
multiplication tensor by simply following the known proofs. However,
for most of these basic tensors with explicit tensor decompositions,
there are matching lower bounds on the rank via the \emph{substitution
method} (e.g., the proof of Proposition
\ref{prop:border-rank-smaller-than-s-rank}). Substitution method lower
bounds also prove lower bounds on s-rank, so there does not seem to be
an opportunity for an easy improvement from switching to s-rank.
However, an improvement by switching to border s-rank might be
possible. As a concrete example, we do not know whether the border
s-rank, or even just the s-rank, of $\langle 2,2,2 \rangle$ is $6$ or
$7$.

\section{Matrix multiplication via coherent configurations}
\label{section:coherent}

In this section, we describe how to embed matrix multiplication into
algebra multiplication.  To do so, it is helpful to have a basis with
considerable combinatorial structure.  As mentioned in the
introduction, adjacency algebras of coherent configurations are a
promising family of algebras to use here.  In
Subsections~\ref{subsec:coherent} and~\ref{subsec:adjacency}, we review
the basic theory of coherent configurations and their adjacency
algebras, and in Subsection~\ref{subsec:embedadj} we specialize our
general theory to this setting.

\subsection{Realizing matrix multiplication in algebras}

We can bound the s-rank of matrix multiplication by restricting
the structural tensor of an algebra.  Let $A$ be a
finite-dimensional complex algebra, and let $u_1,\dots,u_r$ be
a basis for $A$.  Then there are coefficients $\lambda_{i,j,k}$
such that
\[
u_i u_j = \sum_{k=1}^r \lambda_{i,j,k} u_k;
\]
they are called the \emph{structure constants} of $A$ with respect to
this basis.  The \emph{structural tensor} is the trilinear form
\[
\sum_{i,j,k} \lambda_{i,j,k} \vx_i \vy_j \vz_k.
\]
It is isomorphic to the multiplication tensor (i.e., the
element of $A^* \otimes A^* \otimes A$ corresponding to the
multiplication map from $A \otimes A$ to $A$).  More generally,
if we use any three bases $u_1,\dots,u_r$, $v_1,\dots,v_r$, and
$w_1,\dots,w_r$ for $A$ and define the coefficients by
\[
u_i v_j = \sum_{k=1}^r \lambda_{i,j,k} w_k,
\]
then the corresponding tensor is isomorphic to the structural
tensor.

\begin{definition} \label{definition:realize}
Let $A$ be an $r$-dimensional complex algebra with structure constants
$\lambda_{i,j,k}$ corresponding to some choice of bases.  We say $A$
\emph{realizes} $\langle \ell,m,n \rangle$ if there exist three
injective functions
\[
\alpha\colon[\ell]\times[m] \to [r], \quad
\beta\colon[m]\times[n] \to [r], \quad \gamma\colon[n]\times[\ell]
\to [r]
\]
such that
\[
\lambda_{\alpha(a,b'), \beta(b,c'), \gamma(c,a')} \ne 0
\]
if and only if $a = a'$, $b=b'$, and $c = c'$.
\end{definition}

Note that this definition depends on the choice of basis.  We
will typically suppress the choice of basis in the notation,
because the algebras we deal with later in the paper will
always come with a standard basis.

One might reasonably use the term ``s-realize'' instead of ``realize''
in Definition~\ref{definition:realize}.  We have chosen to use the simpler
term, rather than reserving it for strict realization involving only
structure constants that are $0$ or $1$, because we know of few
interesting examples of strict realization beyond group algebras
(where the notions coincide).

\begin{proposition}
If an algebra $A$ realizes $\langle \ell,m,n \rangle$, then the s-rank
of $\langle \ell,m,n \rangle$ is at most the rank of the structural
tensor for $A$.
\end{proposition}

\begin{proof}
Suppose $A$ realizes $\langle \ell,m,n \rangle$ via
$\alpha,\beta,\gamma$, and consider the structural tensor
\[
\sum_{i,j,k} \lambda_{i,j,k} \vx_i \vy_j \vz_k.
\]
Define $\vu_{a,b'} = \vx_{\alpha(a,b')}$, $\vv_{b,c'} =
\vy_{\beta(b,c')}$, and $\vw_{c,a'} = \vz_{\gamma(c,a')}$; furthermore,
set $\vx_i=0$ when $i$ is not in the image of $\alpha$, $\vy_j=0$ when
$j$ is not in the image of $\beta$, and $\vz_k=0$ when $k$ is not in
the image of $\gamma$.  Under this change of variables, the structural
tensor becomes
\[
\sum_{a,a',b,b',c,c'} \lambda_{\alpha(a,b'), \beta(b,c'), \gamma(c,a')} \vu_{a,b'}\, \vv_{b,c'}\, \vw_{c,a'},
% ad hoc spacing
\]
and by assumption the terms vanish unless $a=a'$, $b=b'$, and $c=c'$.
Thus,
\[
\sum_{a,b,c} \lambda_{\alpha(a,b), \beta(b,c), \gamma(c,a)} \vu_{a,b} \,\vv_{b,c}\, \vw_{c,a}
% ad hoc spacing
\]
has rank at most that of the structural tensor.  This new tensor is a
weighting of the matrix multiplication tensor $\langle \ell,m,n
\rangle$, so the s-rank of $\langle \ell,m,n\rangle$ is at most the
rank of the structural tensor.
\end{proof}

Note that this proof in fact gives a very simple algorithm for reducing
a weighted matrix multiplication to an algebra multiplication, along
the lines of the reduction in \cite{CU}.

Recall that an algebra is semisimple if it is a product of matrix
algebras.  In other words, it is semisimple if there are
\emph{character degrees} $d_1,\dots,d_t$ so that
\[
A \cong \C^{d_1 \times d_1} \times \dots \times \C^{d_t \times d_t}.
\]
In that case, the structural tensor is isomorphic to $\langle
d_1,d_1,d_1 \rangle \oplus \dots \oplus \langle d_t,d_t,d_t \rangle$.

\begin{proposition}
If a semisimple algebra $A$ with character degrees $d_1,\dots,d_t$
realizes $\langle \ell,m,n \rangle$, then
\[
(\ell m n)^{\omega_s/3} \le d_1^\omega + \dots + d_t^\omega.
\]
\end{proposition}

This proposition also involves a natural algorithm, which reduces a
weighted matrix multiplication to a collection of unweighted matrix
multiplications.

\begin{proof}
For each $\varepsilon>0$, there is a constant $C$ such that $R(\langle
d,d,d \rangle) \le C d^{\omega+\varepsilon}$ for all $d$.  It follows
that
\[
(\ell m n)^{\omega_s/3} \le R_s(\langle \ell,m,n\rangle) \le C d_1^{\omega+\varepsilon} + \dots +C d_t^{\omega+\varepsilon},
\]
but the $C$ and $\varepsilon$ are problematic.  To remove them, we will
use the trick of computing the asymptotic rank for high tensor powers.
The algebra $A^{\otimes N}$ realizes a weighted version of $\langle
\ell,m,n \rangle^{\otimes N} =\langle \ell^N, m^N, n^N \rangle$, and it
has character degrees given by $N$-fold products $d_{i_1} \dots
d_{i_N}$. Thus,
\[
(\ell m n)^{N \omega_s/3} \le C \left( d_1^{\omega+\varepsilon} + \dots + d_t^{\omega+\varepsilon}\right)^N.
\]
Now taking $N$-th roots and letting $N$ tend to infinity yields
\[
(\ell m n)^{\omega_s/3} \le d_1^{\omega+\varepsilon} + \dots + d_t^{\omega+\varepsilon},
\]
and because this holds for all $\varepsilon>0$ it also holds
for $\varepsilon=0$ by continuity.
\end{proof}

Definition~\ref{definition:realize} generalizes the triple product
property from \cite{CU}.  Recall that three subsets $S,T,U$ of a group satisfy
the \emph{triple product property} if
\begin{gather*}
s^{-1}s't^{-1}t'u^{-1}u' = 1\\
\Leftrightarrow\\
s=s',\ t=t',\ u=u'
\end{gather*}
holds for $s,s' \in S$, $t,t' \in T$, and $u,u' \in U$.
To see why Definition~\ref{definition:realize} is a generalization,
suppose $A$ is the group algebra
of a finite group, and choose the group elements themselves as a basis.
Then $\alpha(a,b')$, $\beta(b,c')$, and $\gamma(c,a')$ correspond to
group elements $g_{a,b'}$, $h_{b,c'}$ and $k_{a',c}$ such that
\begin{equation} \label{eq:defeq}
g_{a,b'} h_{b,c'} = k_{a',c}
\end{equation}
if and only if $a=a'$, $b=b'$, and $c=c'$.  We wish to find group
elements $s_a,t_b,u_c$ such that $g_{a,b} = s_a t_b^{-1}$, $h_{b,c} =
t_b u_c^{-1}$, and $k_{a,c} = s_a u_c^{-1}$;  then
$\{s_a\},\{t_b\},\{u_c\}$ satisfy the triple product property.
To find these group elements, fix $b_0$, and let $s_a =
g_{a,b_0}$ and $u_c = h^{-1}_{b_0,c}$.  Then $s_a u_c^{-1} = k_{a,c}$
automatically. Furthermore, \eqref{eq:defeq} implies that
$g^{-1}_{a,b} g_{a,b_0} = h_{b_0,c} h^{-1}_{b,c}$, with this group
element being independent of $a$ and $c$.  Calling it $t_b$ completes
the construction.

\subsection{Coherent configurations} \label{subsec:coherent}

Coherent configurations are remarkable structures that unify much of
group theory and algebraic combinatorics \cite{Hig1,Hig2,Hig3}. A
\emph{coherent configuration} of rank $r$ is a finite set
$\mathcal{C}$, whose elements are called \emph{points}, with a
partition of $\mathcal{C}^2$ into subsets $R_1,R_2,\dots,R_r$ called
\emph{classes} such that
\begin{enumerate}
\item[(1)] the diagonal $\{(x,x) : x \in \mathcal{C}\}$ is the
    union of some of the classes,

\item[(2)] for each $i \in [r]$ there exists $i^* \in [r]$ such
    that $R_i^* = R_{i^*}$, where
\[
R_i^* = \{(b,a) : (a,b) \in R_i\},
\]
and

\item[(3)] there exist integers $p^k_{i,j}$ for $i,j,k \in [r]$
    such that for all $x,y \in \mathcal{C}$ with $(x,y) \in R_k$,
\[
\#\{z \in \mathcal{C} : \textup{$(x,z) \in R_i$ and $(z,y) \in R_j$}\} = p^k_{i,j}.
\]
\end{enumerate}
We say $\mathcal{C}$ is \emph{symmetric} if $R_i^*=R_i$ for all $i$ and
\emph{commutative} if $p^k_{i,j} = p^k_{j,i}$ for all $i,j,k$.
(Symmetry implies commutativity, but not vice versa.) The numbers
$p^k_{i,j}$ are called the \emph{intersection numbers} of the
configuration.  The configuration is an \emph{association scheme} if
the diagonal is itself one of the classes.
It is easily proved that a commutative coherent configuration must be an association scheme \cite[p.~14]{Hig3}.

Every finite group $G$ defines an association scheme, with $G$ as its
set of points and $G^2$ partitioned into subsets $R_g = \{(h,hg) : h
\in G\}$ with $g \in G$.  Then for $g,h,k \in G$,
\[
p^k_{g,h} = \begin{cases} 1 & \textup{if $gh=k$, and}\\
0 & \textup{otherwise.}
\end{cases}
\]
The intersection numbers encode the multiplication table of the
group, so the group and the corresponding association scheme
are fully equivalent structures.  Note that this association
scheme is commutative iff $G$ is, while it is symmetric iff $g
= g^{-1}$ for all $g \in G$. One can show that an association
scheme comes from a group in this way if and only if all its
intersection numbers are at most $1$.

More generally, suppose $G$ acts on a finite set $X$.  Then
partitioning $X^2$ into the orbits of $G$ under the diagonal
action defines a coherent configuration, called a
\emph{Schurian} coherent configuration. It is an association
scheme iff $G$ acts transitively on $X$.

Many important examples in combinatorics fit into this framework.  For
example, the Hamming scheme consists of the points in $\{0,1\}^n$ with
classes defined by Hamming distance. From a group-theoretic
perspective, it is the Schurian association scheme defined by the
action of the semidirect product $S_n \ltimes (\Z/2\Z)^n$ on
$\{0,1\}^n$, although this formulation is excessive for most purposes.

If $G$ acts transitively on $X$, then we can identify $X$ with
$G/H$, where $H$ is the stabilizer of a point in $X$.  Note
that $G/H$ is not a group unless $H$ is a normal subgroup, but
it is always an association scheme.  In certain cases, called
\emph{Gelfand pairs} $(G,H)$, the quotient $G/H$ is a
commutative association scheme (although the groups $G$ and $H$
will typically not be commutative). For example, this occurs
for the Hamming scheme.

There are also numerous combinatorial examples of association
schemes and coherent configurations that do not come from
symmetry groups.  For example, strongly regular graphs are the
same thing as symmetric association schemes of rank $3$.  More
generally, every distance-regular graph is an association
scheme when the classes are defined by the graph metric.  Some of
these graphs are Schurian association schemes, but many are not.

A \emph{fusion} of a coherent configuration $\mathcal{C}$ is a
configuration $\mathcal{C}'$ with the same set of points and with the
classes of $\mathcal{C}'$ all given by unions of classes of
$\mathcal{C}$.  (Note that this must be done carefully, since taking
arbitrary unions will generally not yield a coherent configuration.)
Another important construction is the \emph{direct product}: given two
coherent configurations $\mathcal{C}$ and $\mathcal{C}'$, their product
$\mathcal{C} \times \mathcal{C}'$ has the direct product of their point
sets as its point set, with the class of $(c_1,c_1')$ and $(c_2,c_2')$
determined by the class of $(c_1,c_2)$ in $\mathcal{C}$ and that of
$(c_1',c_2')$ in $\mathcal{C}'$.  The \emph{symmetric power} $\Sym^k
\mathcal{C}$ is the fusion scheme formed by fusing the classes of the
direct power $\mathcal{C}^k$ under the action of the symmetric group
$S_k$ on the factors.

\subsection{The adjacency algebra} \label{subsec:adjacency}

Every coherent configuration has an associated algebra, which plays the
same role as the group algebra of a group. Let $A_1,\dots,A_r$ be the
adjacency matrices of the relations $R_1,\dots,R_r$.  In other words,
$A_i$ is indexed by $\mathcal{C}$, with
\[
(A_i)_{x,y} = \begin{cases} 1 & \textup{if $(x,y) \in R_i$, and}\\
0 & \textup{otherwise}
\end{cases}
\]
for $x,y \in \mathcal{C}$.  The \emph{adjacency algebra}
$\C[\mathcal{C}]$ of $\mathcal{C}$ is the complex algebra generated by
these adjacency matrices.  (Note that it contains the identity because
the diagonal is a union of classes.) An easy calculation shows that
\[
A_i A_j = \sum_k p^k_{i,j} A_k,
\]
so $\C[\mathcal{C}]$ is spanned by $A_1,\dots,A_r$.  It is a
commutative algebra if and only if $\mathcal{C}$ is
commutative.

The adjacency algebra is closed under the conjugate transpose, so it is
a semisimple algebra (see, for example, Theorem~3.2 in \cite{Cameron}).
Thus, there exist \emph{character degrees}
$d_1,\dots,d_k$ such that
\[
\C[\mathcal{C}] \cong \C^{d_1 \times d_1} \times \dots \times \C^{d_k \times d_k}.
\]
Of course, $d_1^2+\dots+d_k^2$ must equal the dimension of
$\C[\mathcal{C}]$, which is the rank of $\mathcal{C}$.  The adjacency
algebra of a commutative coherent configuration of rank $r$ is
isomorphic to $\C^r$.

The structural tensor of $\C[\mathcal{C}]$ is
\[
\sum_{i,j,k \in [r]} p^k_{i,j} \vx_i \vy_j \vz_k,
\]
but (as we will see shortly) it is often convenient to use
\[
\sum_{i,j,k \in [r]} p^{k^*}_{i,j} \vx_i \vy_j \vz_k
\]
instead.  This isomorphic tensor simply amounts to reordering the
variables $\vz_k$.  Note that the rank $r$ of the coherent
configuration is not necessarily the same as the rank of the structural
tensor: they are equal if and only if the configuration is commutative.

\subsection{Embedding matrix multiplication into an adjacency
algebra} \label{subsec:embedadj}

Let $\mathcal{C}$ be a coherent configuration of rank $r$, with
notation as in the previous subsection.

\begin{definition}
Three classes $i,j,k$ form a \emph{triangle} if there exist
points $x, y, z$ such that $(x,y) \in R_i$, $(y,z) \in R_j$,
and $(z, x) \in R_k$.
\end{definition}

In terms of intersection numbers, classes $i,j,k$ form a triangle iff
$p^{k^*}_{i,j} > 0$.  (Note that we use $k^*$ instead of $k$ to switch
the order of $x$ and $z$.)  This is why we prefer to use $k^*$ instead
of $k$ in the structural tensor: otherwise, the cyclic symmetry among
$x,y,z$ is broken.

\begin{definition}
A coherent configuration ${\cal C}$ of rank $r$ \emph{realizes}
$\langle \ell,m,n \rangle$ if there exist three injective functions
\[
\alpha\colon[\ell]\times[m] \to [r], \quad
\beta\colon[m]\times[n] \to [r], \quad \gamma\colon[n]\times[\ell]
\to [r]
\]
such that $\alpha(a,b'), \beta(b,c'), \gamma(c,a')$ form a triangle iff
$a = a'$, $b=b'$, and $c = c'$.
\end{definition}

Of course this definition assumes a fixed numbering of the
classes in $\mathcal{C}$.  It amounts to the general definition
of realization in an algebra, specialized to our choice of
structural tensor.

\begin{example} \label{example:trivial}
As a simple example, let $\mathcal{C}$ be the coherent
configuration on $n$ points for which every pair of points
defines a distinct class. If we index the classes with pairs of
points, then $(a,b')$, $(b,c')$, and $(c,a')$ form a triangle
if and only if $a=a'$, $b=b'$, and $c=c'$, so $\C[\mathcal{C}]$
trivially realizes $\langle n,n,n \rangle$.  As one might
expect from such a trivial example, the embedding yields no
benefit for matrix multiplication, because in fact
$\C[\mathcal{C}] \cong \C^{n \times n}$.
\end{example}

In the next section we will construct less trivial examples.
In the meantime, we note the following proposition, which works
out the conditions for a coherent configuration arising from a
group action to realize matrix multiplication.

\begin{proposition} \label{proposition:action}
Let $G$ be a finite group acting on a set $X$, and let $\mathcal{C}$ be
the corresponding Schurian coherent configuration. Suppose there exist
subsets $A,B,C \subseteq X$ such that for all $f,g,h \in G$ and all
$a\in A$, $b \in B$, and $c\in C$,
\begin{multline*}
\textup{if $fa \in A$, $gb \in B$, $hc \in C$ and $fgh=1$,}\\
\textup{then $fa=a$, $gb=b$,
and $hc=c$.}
\end{multline*}
Then $\mathcal{C}$ realizes $\langle |A|, |B|, |C| \rangle$.
\end{proposition}

\begin{proof}
Recall that the classes of $\mathcal{C}$ are the orbits of $G$ on
$X^2$.  If we identify $A$ with $[|A|]$, etc., then we realize $\langle
|A|, |B|, |C| \rangle$ via maps $\alpha,\beta,\gamma$ such that for $a
\in A$ and $b' \in B$, $\alpha(a,b')$ is the orbit of $(a,b')$ in
$X^2$, etc.
The map $\alpha$ is injective, because $\alpha(a, b) = \alpha(a', b')$
implies $fa = a'$ and $fb = b'$ for some $f \in G$, and the hypothesis
then implies $a = a'$ and $b = b'$ (take $g=f^{-1}$ and $h=1$);
the maps $\beta$ and $\gamma$ are injective by the same argument.
Now we wish to show that the classes of $(a,b')$, $(b,c')$,
and $(c,a')$ form a triangle if and only if $a=a'$, $b=b'$, and $c=c'$
(where $a,a' \in A$, $b,b' \in B$, and $c,c' \in C$).  Saying they form
a triangle means there exist $x,y,z \in X$ and $s,t,u \in G$ such that
$(x,y) = (sa,sb')$, $(y,z) = (tb,tc')$, and $(z,x) = (uc,ua')$.  If we
set $f = u^{-1}s$, $g =s^{-1}t$, and $h=t^{-1}u$, then $fgh=1$ and $fa
= a'$, $gb = b'$, and $hc=c'$.  Now by hypothesis we have $a=a'$,
$b=b'$, and $c=c'$, as desired.
\end{proof}

If we let $G$ act on itself by left translation, then the hypothesis of
Proposition~\ref{proposition:action} simply asserts that $A$, $B$, and
$C$ satisfy the triple product property from \cite{CU}.  Thus, the
proposition gives a natural generalization of the triple product
property from groups to group actions.

\section{Simultaneous embeddings and symmetric powers}

Our best construction techniques so far are all based on realizing
several independent matrix multiplications simultaneously; this was
called the \emph{simultaneous triple product property} in \cite{CKSU}.
In a coherent configuration, the definition amounts to the following
(and of course one can give an analogous definition in any algebra):

\begin{definition} \label{definition:simultaneous}
A coherent configuration ${\cal C}$ of rank $r$ \emph{realizes}
$\oplus_i \langle \ell_i,m_i,n_i \rangle$ if there exist injective
functions
\begin{align*}
\alpha_i\colon[\ell_i]\times[m_i] \to [r]\\
\beta_i\colon[m_i]\times[n_i] \to [r]\\
\gamma_i\colon[n_i]\times[\ell_i] \to [r]
\end{align*}
such that $\alpha_i(a,b'),
\beta_j(b,c'), \gamma_k(c, a')$ form a triangle iff $i = j = k$ and $a
= a'$, $b=b'$, and $c = c'$.
\end{definition}

If $\mathcal{C}$ realizes $\langle \ell_1,m_1,n_1 \rangle \oplus \dots
\oplus \langle \ell_k,m_k,n_k \rangle$ and $T$ is its structural
tensor, then
\[
R_s\big(\langle\ell_1,m_1,n_1 \rangle \oplus \dots
\oplus \langle \ell_k,m_k,n_k\rangle\big) \le R(T).
\]
One can imitate the proof of the asymptotic sum inequality to show that
\begin{multline}
\label{eqn:s-rank-asi}
\big(\ell_1 m_1 n_1\big)^{\omega_s/3} + \dots + \big(\ell_k m_k n_k \big)^{\omega_s/3}\\
\le R_s\big(\langle\ell_1,m_1,n_1 \rangle \oplus \dots
\oplus \langle \ell_k,m_k,n_k\rangle\big),
\end{multline}
from which one can deduce bounds on $\omega_s$.  Instead, in this
section we will develop an efficient algebraic method to combine these
independent matrix multiplication realizations into one.  It will yield
the same bound, but also show that this bound is achieved by realizing
a single matrix multiplication in a coherent configuration.  First, we give an example.
This example is extremal, because it realizes the direct sum of $n^{1-o(1)}$
copies of $\langle n,n,n \rangle$ via a coherent configuration of rank $n^3$,
and this cannot be done with rank less than $n^{3-o(1)}$ (because the images of
the embeddings must be disjoint).  If the coherent configuration were commutative,
then we could conclude that $\omega=2$, but it is far from commutative.

\begin{example}
Let $\mathcal{C}$ be the coherent configuration corresponding to
the diagonal action of $\Z/n\Z$ on $(\Z/n\Z)^2$, and let $S \subseteq
\Z/n\Z$ be a set of size $|S|=n^{1-o(1)}$ containing no three-term
arithmetic progression \cite{SS}.  We can index the classes in
$\mathcal{C}$ as
\[
R_{(a,b,c)} = \{((s,s+a),(s+a+b,s+a+b+c)) : s \in \Z/n\Z\},
\]
with $a,b,c \in \Z/n\Z$.  Then $\mathcal{C}$ realizes $\oplus_{i\in S}
\langle n,n,n \rangle$ via maps $\alpha_i,\beta_i,\gamma_i$ defined for
$i \in S$ by
\begin{align*}
\alpha_i(x,y) &= (x,i-x,y)\\
\beta_i(y,z) &= (y,i-y,z)\\
\gamma_i(z,x) &= (z,-2i-z,x).
\end{align*}
Specifically, it is not hard to check that
$(x,i-x,y'),(y,j-y,z'),(z,-2k-z,x')$ form a triangle if and only if
$x=x'$, $y=y'$, $z=z'$, and $i+j=2k$ (in which case $i=j=k$ because $S$
contains no three-term arithmetic progressions).  However, this example
does not prove any nontrivial bound on $\omega$, because in fact the
character degrees of $\mathcal{C}$ are all equal to $n$ (repeated $n$
times).
\end{example}

As promised, we now give a constructive proof of
\eqref{eqn:s-rank-asi}. The proof converts a coherent
configuration that realizes several independent matrix
multiplications into a single coherent configuration
that realizes a \emph{single} matrix multiplication. Moreover, the
resulting coherent configuration is commutative if the original
one was. Because the proof actually constructs a coherent
configuration rather than just deducing the bound on
$\omega_s$, we can use it to obtain \emph{commutative} coherent
configurations that prove nontrivial
bounds on $\omega$. This establishes one of the main points of
the paper: that the noncommutativity that was necessary in the
group-theoretic approach can be avoided in the generalization
to coherent configurations.

We also find that a consequence of either of the two main conjectures of \cite{CKSU} is
that \emph{commutative} coherent configurations
suffice to prove $\omega=2$. This raises our hope that one
could find commutative coherent configurations of rank $n^{2+o(1)}$
that realize $\langle n,n,n \rangle$ and thus prove $\omega=2$.

The main idea of the proof is to take symmetric powers, as described next:
\begin{theorem}
\label{thm:symmetrized-subalgebra} Let ${\cal C}$ be a coherent
configuration of rank $r$ that realizes $\oplus_{i = 1}^k \langle
\ell_i,m_i,n_i \rangle$. Then the symmetric power $\Sym^k \mathcal{C}$
realizes $\langle \prod_i \ell_i, \prod_i m_i, \prod_i n_i\rangle$ and
has rank ${{r+ k-1} \choose k}$.
\end{theorem}

\begin{proof}
The classes $R_I$ of the $k$-fold direct product of ${\cal C}$ are
indexed by vectors $I \in [r]^k$. The symmetric group $S_k$ acts on
$[r]^k$ by permuting the $k$ coordinates, and the orbits of this action
naturally correspond to the $k$-multisubsets of $[r]$. Recall that $\Sym^k
\mathcal{C}$ is the fusion configuration with a class for each
orbit; i.e., for each $k$-multisubset $S$ of $r$, we have a class $R_S'$
that is the union of $R_I$ over all $I$ in the orbit corresponding to
$S$. The rank of $\Sym^k \mathcal{C}$ is the number of distinct orbits
(equivalently, the number of distinct $k$-multisubsets of $[r]$), which is
${{r + k-1} \choose k}$.

Set $L = \prod_i \ell_i$, $M= \prod_i m_i$, and $N = \prod_i n_i$. Now,
${\cal C}^k$ realizes $\langle L, M, N \rangle$, so there exist
injective functions
\begin{align*}
\alpha\colon[L]\times[M] \to [r]^k\\
\beta\colon[M]\times[N] \to [r]^k\\
\gamma\colon[N]\times[L]\to [r]^k
\end{align*}
satisfying the conditions of Definition~\ref{definition:realize};
specifically,
\begin{align*}
\alpha(A, B) &=  (\alpha_1(A_1,B_1),\dots,\alpha_k(A_k,B_k)) \\
\beta(B, C) &=  (\beta_1(B_1,C_1),\dots,\beta_k(B_k,C_k)) \\
\gamma(C,A) &=  (\gamma_1(C_1,A_1),\dots,\gamma_k(C_k,A_k)),
\end{align*}
where $\mathcal{C}$ realizes $\oplus_{i = 1}^k \langle \ell_i,m_i,n_i
\rangle$ via $\alpha_i,\beta_i,\gamma_i$.

We claim that in fact $\alpha, \beta, \gamma$ are injective even in the
fusion configuration, where we collapse the orbits. For suppose that
$\alpha(A,B)= \pi \alpha(A', B')$ for some $\pi \in S_k$. Then $\pi$
must be the identity since the maps $\alpha_i$ have disjoint images in
$[r]$ (this follows immediately from
Definition~\ref{definition:simultaneous}), and then by injectivity of
$\alpha_i$ we have $(A, B) = (A',B')$. The same holds for $\beta$ and
$\gamma$.

Moreover, the orbits of $\alpha(A,B'), \beta(B,C'), \gamma(C, A')$ form
a triangle in $\Sym^k \mathcal{C}$ iff $A = A'$, $B= B'$, and $C= C'$.
For suppose there exist points $X,Y,Z$ and permutations $\pi_1, \pi_2,
\pi_3 \in S_k$ for which $(X,Y) \in R_I$, $(Y,Z) \in R_J$ and $(Z,X)
\in R_K$, where $I = \pi_1\alpha(A,B')$, $J = \pi_2 \beta(B,C')$, and
$K = \pi_3\gamma(C, A')$. Then we must have $\pi_1 = \pi_2 = \pi_3$
because $\alpha_i(A_i, B_i'), \beta_j(B_j,C_j'), \gamma_k(C_k, A_k')$
cannot form a triangle unless $i = j = k$, and then $A' = A$, $B=B'$
and $C = C'$ follow from the fact that these equalities hold for each
coordinate (by properties of $\alpha_i, \beta_i, \gamma_i$). Thus
$\Sym^k \mathcal{C}$ realizes $\langle L, M, N\rangle$, as claimed.
\end{proof}

\begin{corollary}
\label{cor:asi-geometric-mean-version} Let ${\cal C}$ be a commutative
coherent configuration of rank $r$ that realizes $\oplus_{i = 1}^k
\langle \ell_i,m_i,n_i \rangle$. Then symmetric powers of direct powers
of ${\cal C}$ prove the bound
\[k\cdot\left (\left (\prod_{i=1}^k \ell_im_in_i\right )^{1/k}\right )^{\omega_s/3} \le r.\]
\end{corollary}

More precisely, they come arbitrarily close to this bound.

\begin{proof}
By taking direct powers, ${\cal C}^t$ realizes
\[
\bigoplus_{I \in [k]^t}
\big\langle \prod_{j \in [t]}\ell_{I_j}, \prod_{j \in [t]}m_{I_j},\prod_{j
\in [t]}n_{I_j} \big\rangle.
\]
Setting $L = \prod_{i \in k} \ell_i$, $M =
\prod_{i \in k} m_i$ and $N = \prod_{i \in k} n_i$, we have by
Theorem~\ref{thm:symmetrized-subalgebra} that $\Sym^k {\cal C}^t$
realizes \[\langle L^{tk^{t-1}}, M^{tk^{t-1}}, N^{tk^{t-1}}\rangle\] and
has rank ${{r^t+ k^t-1} \choose k^t}$. By Proposition
\ref{prop:s-rank-omega-bound} we have
\begin{align*}
(LMN)^{\omega_stk^{t-1}/3} &\le {{r^t+ k^t-1} \choose k^t}\\
&\le \left (\frac{e(r^t+k^t-1)}{k^t} \right )^{k^t}\\
&\le \left (\frac{2er^t}{k^t}\right)^{k^t},
\end{align*}
where the last inequality uses the fact that $k \le r$. Taking
$tk^t$-th roots and letting $t$ go to infinity, we obtain the bound
$(LMN)^{\omega_s/(3k)} \le r/k.$
\end{proof}

By weighting the independent matrix multiplications
appropriately, we find that the geometric mean can be replaced
by the arithmetic mean, to obtain a bound on $\omega_s$
identical to the asymptotic sum inequality \eqref{eqn:asi}:

\begin{theorem} \label{theorem:asi}
Let ${\cal C}$ be a commutative coherent configuration of rank $r$ that
realizes $\oplus_{i = 1}^k \langle \ell_i,m_i,n_i \rangle$. Then
symmetric powers of direct powers of ${\cal C}$ prove the bound $\sum_i
(\ell_im_in_i)^{\omega_s/3} \le r$.
\end{theorem}

\begin{proof}
Fix an integer $N$ and $\mu =(\mu_1, \ldots, \mu_k)$ satisfying $\mu_i
\ge 0$ and $\sum_i \mu_i = N$. Then the direct product ${\cal C}^N$
realizes $L = {N \choose \mu}$ independent copies of $\langle
\prod_i\ell_i^{\mu_i}, \prod_im_i^{\mu_i}, \prod_in_i^{\mu_i} \rangle$
(the key is that now these are all the same size). Applying Corollary
\ref{cor:asi-geometric-mean-version}, we find that symmetric powers of
direct powers of ${\cal C}$ prove the bound
\begin{equation}
\label{eq:asi-for-given-distribution}
{N \choose \mu}\prod_i(\ell_im_in_i)^{\mu_i\omega_s/3} \le r^N.
\end{equation}
Summing this inequality over all $\mu$ gives
\[\left (\sum_i(\ell_im_in_i)^{\omega_s/3} \right )^N \le {{N+k-1} \choose k-1} \cdot r^N,\]
and the theorem follows by taking $N$-th roots and letting $N$ go to
infinity. Note that for each $N$, by an averaging argument, there must
be a particular distribution $\mu$ for which the left hand side of
\eqref{eq:asi-for-given-distribution} is at least $\left
(\sum_i(\ell_im_in_i)^{\omega_s/3} \right )^N/{{N+k-1} \choose k-1}$
and this is a concrete sequence of coherent configurations that prove
the same bound in the limit.
\end{proof}

The results of this section are not specific to coherent
configurations.  Given any algebra $A$, the analogous construction is
to look at the subalgebra of $A^{\otimes n}$ invariant under the action
of $S_n$.

Before using Theorem \ref{theorem:asi} to obtain bounds on
$\omega$, we briefly contrast other proofs of the asymptotic
sum inequality with the above proof, which seems structurally
different as we now explain. The standard proof of the
asymptotic sum inequality takes a tensor $T$ realizing
$\oplus_{i=1}^k \langle \ell_i,m_i,n_i \rangle$ and finds $k!$
independent copies of $\langle \prod_i \ell_i, \prod_i m_i,
\prod_i n_i \rangle$ in $T^{\otimes k}$. By performing block
matrix multiplication, these are capable of realizing the {\em
larger} matrix multiplication instance $\langle K\cdot\prod_i
\ell_i, K\cdot\prod_i m_i, K\cdot\prod_i n_i \rangle$ (where $K
\approx k!^{1/\omega}$), and the general bound follows after
some manipulations analogous to our Corollary
\ref{cor:asi-geometric-mean-version} and Theorem
\ref{theorem:asi}. In contrast, our proof finds a {\em single}
copy of $\langle \prod_i \ell_i, \prod_i m_i \prod_i n_i
\rangle$ in $T^{\otimes k}$, and then uses the fact that $T$
has special structure---it is the structural tensor of an
algebra---to argue that the same matrix multiplication instance
survives after symmetrizing the $k$\nobreakdash-th power. Symmetrizing
reduces the rank, and thus the s-rank actually {\em shrinks}
enough to obtain the same bound. We know of no other proof that
works by shrinking the rank (including the proof in~\cite{CKSU}
for independent matrix multiplications realized in group
algebras).

\subsection{Nontrivial bounds on $\omega$}

Using Theorem~\ref{theorem:asi}, we can convert all of the
results of \cite{CKSU} into realizations of a single matrix
multiplication tensor in a \emph{commutative} coherent
configuration, namely a symmetric power of an abelian group.
While the starting constructions are not new, the final
algorithms are (they use machinery introduced in this paper).
They establish that commutative coherent configurations suffice
to prove nontrivial bounds on $\omega$, and even point to a
specific family of commutative coherent configurations that we
conjecture is capable of proving $\omega = 2$.

\begin{theorem}
\label{thm:cksu-bounds-on-omega} There exist commutative
coherent configurations that prove s-rank exponent bounds
$\omega_s \le 2.48$, $\omega_s \le 2.41$, and $\omega_s \le
2.376$, and thus corresponding exponent bounds $\omega \le
2.72$, $\omega \le 2.62$, and $\omega \le 2.564$, respectively.
\end{theorem}

\begin{proof}
Apply Theorem \ref{theorem:asi} to the abelian
group constructions of Proposition~3.8 in \cite{CKSU},
Theorems~3.3 and~6.6 in \cite{CKSU}, and the generalization
matching \cite{CW} (as stated in \cite{CKSU} but not described
in detail), respectively. Each of these constructions from \cite{CKSU},
when viewing groups as coherent configurations and adopting the language
of this paper, gives coherent configurations satisfying
Definition~\ref{definition:simultaneous}. Apply Theorem
\ref{thm:remove-weights} to the resulting s-rank exponent
bounds to obtain the claimed bounds on $\omega$.
\end{proof}

The specific exponent bounds cited above of course all suffer
from the 50\% penalty introduced by Theorem
\ref{thm:remove-weights}. But the numbers themselves should not
obscure the main point, which is that matrix multiplication via
coherent configurations is a viable approach to proving $\omega
= 2$. Indeed, we conjecture that commutative coherent
configurations are sufficient to prove $\omega=2$:

\begin{conjecture}
\label{conj:new-conjecture} There exist commutative coherent
configurations ${\cal C}_n$ realizing $\langle n,n,n \rangle$
and of rank $n^{2 + o(1)}$.
\end{conjecture}

Such a family of commutative coherent configurations would prove
$\omega = 2$. If Conjecture~3.4 or~4.7 from \cite{CKSU} hold, then
Conjecture~\ref{conj:new-conjecture} holds via
Theorem~\ref{theorem:asi}.
We note that recent work of Alon, Shpilka, and Umans
\cite{ASU} shows that Conjecture~3.4 from \cite{CKSU} contradicts a
sunflower conjecture, although there is no strong consensus that this
particular sunflower conjecture is true. Among the various
combinatorial/algebraic conjectures implying $\omega = 2$, Conjecture
\ref{conj:new-conjecture} is the weakest (it is implied by the others),
which makes it the ``easiest'' among these potential routes to proving
$\omega =2$.

\section{Families of coherent configurations}

In this section we discuss the suitability of broad classes of coherent
configurations for proving bounds on $\omega$.

\subsection{Coherent configurations with many fibers}

By property~(1) in the definition of a coherent configuration, there is
a subset of the classes that form a partition of the diagonal, and we
call these classes the \emph{fibers} of the coherent configuration. We
noted in Section \ref{subsec:coherent} that coherent configurations
with more than one fiber are noncommutative. More interestingly for our
application, we will see shortly that $n$ fibers suffice to embed $n
\times n$ matrix multiplication.  This observation generalizes
Example~\ref{example:trivial}.

The fibers of a coherent configuration $\mathcal{C}$ correspond to a
partition of the points into subsets
$\mathcal{C}_1,\dots,\mathcal{C}_n$.  Then it follows from property~(3)
in the definition that the classes of $\mathcal{C}$ form a refinement
of the subsets $\mathcal{C}_i \times \mathcal{C}_j$.

\begin{proposition}
Every coherent configuration with $n$ fibers realizes $\langle n,n,n
\rangle$.
\end{proposition}

\begin{proof}
Let $\mathcal{C}$ be a coherent configuration with $n$ fibers and
corresponding partition $\mathcal{C}_1,\dots,\mathcal{C}_n$, and let
$x_1,\dots,x_n$ be a system of distinct representatives for
$\mathcal{C}_1,\dots,\mathcal{C}_n$. Define $\alpha(a, b)$ to be the
class of ${\cal C}$ containing $(x_a, x_{b})$,  $\beta(b, c)$ to be the
class containing $(x_b, x_{c})$ and $\gamma(c, a)$ to be the class
containing $(x_c, x_{a})$. It is easy to verify that these functions
satisfy Definition~\ref{definition:realize}.
\end{proof}

However, this generic embedding does not lead to nontrivial bounds on
$\omega_s$, because a similar argument shows that one of the
character degrees must be at least $n$.  Let $\mathcal{C}'$ be the
coherent configuration with the same points as $\mathcal{C}$ and
$\mathcal{C}_i \times \mathcal{C}_j$ as its classes.  Then the
adjacency algebra of $\mathcal{C}'$ is a subalgebra of that of
$\mathcal{C}$ (the adjacency matrix for $\mathcal{C}_i \times
\mathcal{C}_j$ is the sum of the adjacency matrices for the classes of
$\mathcal{C}$ contained in $\mathcal{C}_i \times \mathcal{C}_j$), and
it is not hard to check that the adjacency algebra of $\mathcal{C}'$ is
isomorphic to $\C^{n \times n}$.  However, $\C^{n \times n}$ cannot be
isomorphic to a subalgebra of a semisimple algebra $\C^{d_1 \times d_1}
\times \dots \times \C^{d_k \times d_k}$ unless $d_i \ge n$ for some
$i$.  To see why, note that projection onto the factors in the
semisimple algebra would yield representations of dimension $d_i$ for
$\C^{n \times n}$. Because $\C^{n \times n}$ is a simple algebra, these
projections must vanish unless $d_i \ge n$.

\subsection{Schurian coherent configurations}

Recall that the \emph{Schurian} coherent configurations are those
obtained from actions of groups on sets, and that such a coherent
configuration is an association scheme iff the action is transitive.

In the special case of a finite group acting on itself by right
multiplication, our framework is equivalent to the triple product
property from \cite{CU}.  Thus, the conjectures of \cite{CKSU}, which
all imply $\omega = 2$ via the triple product property in groups, imply
that Schurian coherent configurations are sufficient to achieve
$\omega_s = 2$.

More interestingly, observe that the coherent configurations arising in
Theorem~\ref{thm:cksu-bounds-on-omega} and
Conjecture~\ref{conj:new-conjecture} are in fact commutative Schurian
coherent configurations. This is because symmetric powers of coherent
configurations arising from abelian groups, which are commutative, are
Schurian via a wreath product action: if $G$ is a group and ${\cal C}$
is the associated coherent configuration, then $\Sym^k{\cal C}$ is the
Schurian coherent configuration arising from $S_k \ltimes G^k$ acting
on $G^k$.

Thus commutative Schurian coherent configurations, which arise from
transitive group actions, already prove nontrivial bounds on $\omega_s$
(and $\omega$), and if either of the two conjectures in \cite{CKSU} is
true, then they suffice to prove $\omega_s = 2$.

\subsection{Group association schemes}

Another generic way to obtain a Schurian coherent configuration is to
consider $G \times G$ acting on $G$ via $(x,y)\cdot g = xgy^{-1}$. This
gives rise to a commutative coherent configuration (regardless of
whether $G$ is commutative or not, which is attractive for our
application) called the {\em group association scheme}, whose classes
are identified with conjugacy classes of $G$ (i.e., class $R_i =
\{(g,h) : gh^{-1} \in C_i\}$, where $C_i$ is the $i$-th conjugacy
class). Here we show that group association schemes suffice to prove
nontrivial bounds on $\omega_s$, and that if either of the two
conjectures in \cite{CKSU} is true, then group association schemes
suffice to prove $\omega_s = 2$.

We need the following definition from \cite{CKSU} (Definition~5.1):

\begin{definition}
We say that $n$ triples of subsets $A_i, B_i, C_i$ of a group $H$
satisfy the \emph{simultaneous triple product property} if
\begin{gather*}
a_i^{-1}a_j'b_j^{-1}b_k'c_k^{-1}c_i' = 1\\
\Leftrightarrow\\
i = j = k \mbox{ and } a_i = a_j',\ b_j = b_k',\ c_k =
c_i'
\end{gather*}
holds for all $i,j,k$ and $a_i \in A_i$, $a_j' \in A_j$, $b_j \in B_j$,
$b_k' \in B_k$, $c_k \in C_k$, $c_i' \in C_i$.
\end{definition}

When this holds, the coherent configuration associated with the right
action of $H$ on itself realizes $\oplus_i \langle |A_i|, |B_i|, |C_i|
\rangle$ via functions $\alpha_i$, $\beta_i,$ and $\gamma_i$ defined on $A_i\times B_i$,
$B_i \times C_i$, and $C_i \times A_i$, respectively.  Specifically,
$\alpha_i(a,b)$ is the class containing the pair $(a,b)$, etc.  Then
Definition~\ref{definition:simultaneous} amounts to the simultaneous
triple product property.

The paper \cite{CKSU} describes the following constructions,
among others:

\begin{enumerate}
\item[(1)] It follows from Theorem~3.3 and Section~6.3 in \cite{CKSU}
    that for all $m > 2$, and $\ell$ sufficiently large, there are $n$
    triples $A_i,B_i,C_i$ of subsets of $(\Z/m\Z)^{3\ell}$
    satisfying the simultaneous triple product property with $n =
    (27/4)^{\ell - o(\ell)}$ and $|A_i||B_i||C_i| = (m-2)^{3\ell}$
    for all $i$.  Applying Theorem~\ref{theorem:asi} and
    taking the limit as $\ell \to \infty$ yields
    \[
    \omega_s \le \frac{3 \log m - \log (27/4)}{\log (m-2)},
    \]
    which is optimized by $m=10$ (giving $\omega_s \le 2.41$).

\item[(2)] Either of the two conjectures in \cite{CKSU} implies the
    existence of subsets satisfying the simultaneous triple
    product property in an abelian group $H$ with
    \[
    |A_i| = |B_i| = |C_i| = t \ge n^\varepsilon
    \]
    for $1 \le i \le n$ and $|H| = (t^2 n)^{1 +
    o(1)}$ (as $n \to \infty$ with $\varepsilon>0$ fixed),
    which would prove $\omega=2$.
\end{enumerate}

\begin{theorem}
\label{thm:grp-as} Let $H$ be an abelian group, and suppose $n$ triples
of subsets $A_i, B_i, C_i$ in $H$ satisfy the simultaneous triple
product property. Let $G = S_n \ltimes H^n$, and define
\begin{gather*}
A = A_1 \times A_2 \times \cdots \times A_n\\
B = B_1 \times B_2 \times \cdots \times B_n\\
C = C_1 \times C_2 \times \cdots \times C_n,
\end{gather*}
viewed as subsets of $G$ via the natural embedding of $H^n$ in $G$. Let
$\mathcal{C}$ be the group association scheme of $G$. Then the subsets
$A, B, C$ satisfy the requirements of
Proposition~\ref{proposition:action} with respect to $\mathcal{C}$
(i.e., for the action of $G \times G$ on $G$), so $\mathcal{C}$
realizes $\langle |A|, |B|, |C| \rangle$.
\end{theorem}

\begin{proof}
We will write elements of $G$ as $h\pi$, with $h \in H^n, \pi \in S_n$,
and we will use $\pi \cdot h$ to denote the permutation action of $S_n$
on $H^n$.  The semidirect product satisfies $\pi h = (\pi \cdot h) \pi$
for $\pi \in S_n$ and $h \in H^n$.

Suppose we have $f = (f_1, f_2)$, $g = (g_1, g_2)$, $h = (h_1, h_2)$ in
$G \times G$ and $a, a' \in A$, $b,b' \in B$, $c,c'
\in C$ for which
\begin{equation} \label{eq:simul1}
\begin{aligned}
fgh &= 1\\
f_1af_2^{-1} & =   a'\\
g_1bg_2^{-1} & =   b'\\
h_1ch_2^{-1} & =   c'.
\end{aligned}
\end{equation}
We wish to conclude that $a = a', b = b', c = c'$. From the latter three
equations in \eqref{eq:simul1}, we see that $f_1 = x_1\pi$, $f_2 = x_2\pi$ for some
$x_1, x_2 \in H^n$ and $\pi \in S_n$.  Similarly, $g_1 = y_1\rho$, $g_2
= y_2\rho$ and $h_1 = z_1\tau$, $h_2 = z_2\tau$. Now, using the
commutativity of $H$, the three equations become
\begin{equation} \label{eq:simul2}
\begin{aligned}
x_1x_2^{-1}& =   a'(\pi \cdot a^{-1}) \\
y_1y_2^{-1} & =   b'(\rho \cdot b^{-1}) \\
z_1z_2^{-1} & =  c'(\tau \cdot c^{-1}).
\end{aligned}
\end{equation}
From $fgh=1$ we have $f_1g_1h_1 = 1$, which implies
\[
1 = f_1g_1h_1 = x_1(\pi\cdot y_1)((\pi\rho) \cdot z_1).
\]
Similarly, from $fgh=1$ we have $f_2g_2h_2 = 1$ and hence
\[
1 = f_2g_2h_2 = x_2(\pi\cdot y_2)((\pi\rho) \cdot z_2).
\]
Using the commutativity of $H$, we obtain from these two equations
\[
x_1x_2^{-1}(\pi\cdot (y_1y_2^{-1}))((\pi\rho) \cdot (z_1z_2^{-1})) = 1,
\]
which combined with \eqref{eq:simul2} yields
\[
a'(\pi \cdot a^{-1})(\pi \cdot b')((\pi\rho) \cdot b^{-1})((\pi\rho) \cdot c')((\pi\rho\tau) \cdot c^{-1}) = 1.
\]
Now, $fgh=1$ implies $\pi\rho\tau = 1$, so we obtain
\[
(\pi \cdot (a^{-1}b'))((\pi\rho) \cdot (b^{-1}c'))(c^{-1}a') = 1.
\]
This implies via the simultaneous triple product property that $\pi =
\pi\rho = 1$. We conclude that $\pi = \rho = \tau = 1$, and then that
$a = a'$, $b = b'$, $c = c'$ as desired.
\end{proof}

To determine what bounds on $\omega_s$ can be expected, we need to know
the rank of this group association scheme, i.e., the number of conjugacy
classes in $S_n \ltimes H^n$:

\begin{lemma} \label{lemma:countconjugacy}
There is a constant $C$ such that for every abelian group $H$,
if $n \le |H|$ then the number of conjugacy classes of $S_n
\ltimes H^n$ is at most $C^n|H|^n/n^n$.
\end{lemma}

This is a crude bound, but it will suffice for our purposes.

\begin{proof}
It is not difficult to prove the following description of
the conjugacy classes in the group
$S_n \ltimes H^n$.  Each
element of this group can be written as $h\pi$ with $h \in H^n$ and
$\pi \in S_n$. The cycle type of $\pi$ is preserved under conjugation,
and the sum of the elements of $H$ in the coordinates of $h$
corresponding to each cycle of $\pi$ is also preserved.  Furthermore,
these invariants completely specify the conjugacy class.  Thus, each conjugacy
class is specified by a multiset of pairs consisting of a cycle
length and an element of $H$, where the cycle lengths must sum to $n$.

The possible cycle types correspond to partitions of $n$, and the number of
them grows subexponentially as $n \to \infty$.  More elementarily,
there are $2^{n-1}$ compositions of $n$ (ways of writing $n$ as an
ordered sum of positive integers), and therefore at most $2^{n-1}$
partitions of $n$.

Suppose the permutation has $c_i$ cycles of length $i$, with
$\sum_{i=1}^n i c_i = n$.  Then there are
\[
\prod_{i=1}^n \binom{|H|+c_i-1}{c_i}
\]
ways to choose the elements of $H$ corresponding to these cycles. Thus,
bounding the number of conjugacy classes in $G$ amounts to bounding how
large this product can be.

We have
\begin{align*}
\prod_{i=1}^n \binom{|H|+c_i-1}{c_i} &\le \prod_{i=1}^n \left(\frac{e(|H|+c_i-1)}{c_i}\right)^{c_i}\\
& \le (2e)^n \prod_{i=1}^n \frac{|H|^{c_i}}{c_i^{c_i}}\\
& \le (2e)^n \prod_{i=1}^n \frac{|H|^{ic_i}}{c_i^{c_i}n^{(i-1)c_i}}\\
& \le (2e)^n \frac{|H|^n}{n^n} \prod_{i=1}^n \left(\frac{n}{c_i}\right)^{c_i}.
\end{align*}
If we set $x_i = c_i/n$, then
\[
\prod_{i=1}^n \left(\frac{n}{c_i}\right)^{c_i} = e^{-n \sum_{i=1}^n x_i \log x_i},
\]
where $\log$ denotes the natural logarithm.
Thus, to complete the proof we must show that $-\sum_{i=1}^n
x_i \log x_i$ is bounded independently of $n$, whenever $x_i
\ge 0$ and $\sum_{i=1}^n i x_i = 1$.  The maximum can be found
using Lagrange multipliers.  One must deal with the
boundary cases when $x_i=0$ for some $i$, and we provide the
details below.

Suppose $x_1,\dots,x_n$ maximize $-\sum_{i=1}^n
x_i \log x_i$ subject to $\sum_{i=1}^n i x_i = 1$ and $x_i \ge 0$.
The desired result is trivial when only one of $x_1,\dots,x_n$
is nonzero. Otherwise, let $z_1,\dots,z_m$ be the nonzero
elements among $x_1,\dots,x_n$.  The equation $\sum_{i=1}^n i x_i = 1$.
becomes $\sum_{i=1}^m y_iz_i=1$, where $y_1 < y_2
< \dots < y_m$ are positive integers.
Then there is a Lagrange multiplier $\lambda$ such that
$-1-\log z_i = \lambda y_i$ for all $i$, and hence
\[
-\sum_{i=1}^m z_i \log z_i = \sum_{i=1}^m z_i (\lambda y_i+1) = \lambda + \sum_{i=1}^m z_i \le \lambda+1.
\]
To bound $\lambda$, note that $z_i = e^{-1-\lambda y_i}$ and
hence
\[
\sum_{i=1}^m y_i e^{-1-\lambda y_i} = 1,
\]
while for $\lambda>1$ we have
\[
\sum_{i=1}^m y_i e^{-1-\lambda y_i} < \sum_{j=1}^\infty j e^{-1-j} < 1.
\]
Thus, $\lambda \le 1$ and so $-\sum_{i=1}^m z_i \log z_i \le 2$.  Combining the
estimates so far shows that we can take $C = 4e^3$ in the lemma
statement.  (The best possible constant is of course much
smaller.)
\end{proof}

This bound on the rank of these group association schemes is precisely
what is needed to recover the desired bounds from the simultaneous
triple product property constructions listed
earlier in this section.  E.g., applying Theorem~\ref{thm:grp-as} and
Lemma~\ref{lemma:countconjugacy} to the second example yields
\[
t^{n \omega_s} \le C^n \frac{(t^2 n)^{n(1 +
    o(1))}}{n^n},
\]
which is equivalent to
\[
\omega_s \log t \le \log C + (2+o(1))\log t + o(\log n).
\]
Because $t \ge n^\varepsilon$, we get $\omega_s=2$ in the limit as $n \to \infty$.

\begin{corollary}
There exist group association schemes that prove $\omega_s \le 2.41$
(and hence $\omega\le2.62$). If either of the conjectures in
\cite{CKSU} is true, then there are group association schemes that
prove $\omega_s=2$ (and hence $\omega=2$).
\end{corollary}

More generally, one can imitate the transition from
Corollary~\ref{cor:asi-geometric-mean-version} (which is analogous to Theorem~\ref{thm:grp-as})
to Theorem~\ref{theorem:asi} to give a proof of Theorem~5.5 from \cite{CKSU}
using group association schemes.

\end{document}